\documentclass[12 pt]{article}
\usepackage{amsthm}
\usepackage{amsmath}
\usepackage{amssymb}
\usepackage[all]{xy}

\newtheorem{thm}{Theorem}[section]

\newtheorem{prop}[thm]{Proposition}

\newtheorem{rem}[thm]{Remark}
\numberwithin{equation}{section}

\newtheorem{theo}{Theorem}[subsection]
\newtheorem{coro}[theo]{Corollary}
\newtheorem{lema}[theo]{Lemma}
\newtheorem{propo}[theo]{Proposition}

\newtheorem{defnt}[theo]{Definition}

\newtheorem{remk}[theo]{Remark}

\begin{document}

\title {On the classification of
unstable $H^{\ast}V-A$-modules}

\author{Dorra BOURGUIBA
\footnote{\fontfamily{cm}\fontsize{8}{8pt}
\selectfont \noindent
pris en charge par l'unit\'e de recherche 00/UR/15-05.
Facult\'e de Sciences--Math\'ematiques, Universit\'e de Tunis, TN-1060 Tunis,
Tunisie.
\it e-mail: \tt dorra.bourguiba@fst.rnu.tn}}
\date{}

\maketitle

\begin{abstract} In this work, we begin studying the classification, up to isomorphism,
of unstable $\mathrm{H}^{\ast}V-A$-modules $E$ such that $\mathbb{F}_{2} \otimes
_{\mathrm{H}^{\ast}V} E $ is isomorphic to a given unstable $A$-module $M$. In fact
this classification depends on the structure of $M$ as unstable
$A$-module. In this paper, we are interested in the case $M$ a nil-closed
unstable $A$-module and the case $M$ is isomorphic to $\sum^{n}\mathbb{F}_{2}$.
We also study, for $V=\mathbb{Z}/2\mathbb{Z}$, the case $M$ is the Brown-Gitler module $\mathrm{J}(2)$.
\end{abstract}
\maketitle
\section{Introduction} Let $V$ be an elementary abelian 2-group of rank $d$,
that is a group isomorphic to $ (\mathbb{Z}/2 \mathbb{Z})^{d}, \; d
\in \mathbb{N}$, $BV$ be a classifying space for the group $V$ and
$\mathrm{H}^{\ast}V = H^{\ast}(BV;\mathbb{F}_{2})$. We recall that $\mathrm{H}^{\ast}V$ is an
$\mathbb{F}_{2}$-polynomial algebra $\mathbb{F}_{2}[t_{1},\ldots,t_{d}]$
on $d$ generators $t_{i},1 \leq i \leq d$, of degree one.
\medskip\\
Let $A$ be the mod.2 Steenrod algebra and $\mathcal{U}$ the category of
unstable $A$-modules. We recall that $\mathrm{H}^{\ast}V-\mathcal{U}$ is the category
whose objects are unstable $\mathrm{H}^{\ast}V-A$-modules and morphisms are
$\mathrm{H}^{\ast}V$-linear and $A$-linear maps of degree
zero. For example, the mod.2 equivariant cohomology of a $V$-CW-complex, which is
the cohomology of the Borel construction, is
an unstable $\mathrm{H}^{\ast}V-A$-module.
\medskip\\
Let $E$ be an unstable $\mathrm{H}^{*}V-A$-module, we denote by $\overline{E}$ the
unstable $A$-module $\mathbb{F}_{2}
\otimes_{\mathrm{H}^{\ast}V}E = E / \widetilde{\mathrm{H}^{*}}V.E$, where
$\widetilde{\mathrm{H}^{*}}V$ denotes
the augmentation ideal of $\mathrm{H}^{*}V$ .
\medskip\\
We have the following problem:
\medskip
\begin{center}
\hspace{1cm}{\bf {$\mathcal{(P)}$ : Let $M$ be an unstable A-module.\\ Classify,
up to isomorphism, unstable $\mathrm{H}^{\ast}V-A$-modules\\ such that
$\overline{E} \cong M$ (as unstable $A$-modules).}}
\end{center}
It is clear that, for every subgroup $W$ of $V$, the unstable
$\mathrm{H}^{\ast}V-A$-module: $$\mathrm{H}^{\ast}W \otimes M$$
is a solution for the problem $\mathcal{(P)}$.
\medskip\\
For $W=0$, a solution of $\mathcal{(P)}$ is given by the unstable
$\mathrm{H}^{\ast}V-A$-module $M$ which is trivial as an $\mathrm{H}^{\ast}V$-module.
\medskip\\
For $W=V$, a solution of $\mathcal{(P)}$ is given by the unstable
$\mathrm{H}^{\ast}V-A$-module $\mathrm{H}^{\ast}V \otimes M$ which is free as an
$\mathrm{H}^{\ast}V$-module.
\medskip\\
If $V=\mathbb{Z}/2\mathbb{Z}$ and $M=\Sigma N$ a suspension of an
unstable $A$-module $N$, then we have, at least, the following two solutions of the
problem $\mathcal{(P)}$ which are free as $H^{*}(\mathbb{Z}/2\mathbb{Z})$-modules:
\begin{enumerate}
    \item $\Sigma (H^{*}(\mathbb{Z}/2\mathbb{Z}) \otimes N)$.
    \item $((H^{*}(\mathbb{Z}/2\mathbb{Z}) ^{\geq 1}) \otimes N$.
\end{enumerate}
These two solutions are different as unstable $A$-modules (here
$H^{*}(\mathbb{Z}/2\mathbb{Z}) ^{\geq 1}$
is the sub-algebra of $H^{*}(\mathbb{Z}/2\mathbb{Z})$ of elements of
degree bigger than or equal to one). This shows that the solutions of the problem
$\mathcal{(P)}$ i.e. the classification, up to isomorphism, of unstable
$\mathrm{H}^{\ast}V-A$-modules such that $\overline{E} \cong M$ (as unstable
$A$-modules), depends on the structure of $E$ as an $\mathrm{H}^{\ast}V$-module and on
the structure of $M$ as unstable $A$-module.
\medskip\\
In this paper we will discuss the solutions of $\mathcal{(P)}$ if $M$ is a nil-closed unstable
$A$-module and $E$ is free as an $H^{*}V$-module and the solutions of
$ \mathcal{(P)}$ if $M$ is isomorphic to $\sum^{n}\mathbb{F}_{2}$ or to $ \mathrm{J}(2)$ and $E$ is free as an $H^{*}V$-module  .
\medskip\\
We begin by proving the following result (which is solution of $( \mathcal{P})$ when $M$ is a nil-closed
unstable $A$-module ).
\begin{thm} Let $E$ be unstable $\mathrm{H}^{\ast}V-A$-module which is free as an
$\mathrm{H}^{\ast}V$-module.
If $\overline{E}$ is a nil-closed unstable $A$-module, then there exists
two reduced $\mathcal{U}$-injectives $I_{0}, \; I_{1}$ and an $\mathrm{H}^{\ast}V-A$-linear
map\\
$ \varphi:  \mathrm{H}^{\ast}V \otimes I_{0} \rightarrow \mathrm{H}^{\ast}V \otimes I_{1}$
such that:
\begin{enumerate}
\item $E \cong ker \varphi$
\item $\overline{E} \cong ker \overline{\varphi}$
\end{enumerate}
\end{thm}
The proof of this result is based on the classification of
$\mathrm{H}^{\ast}V-\mathcal{U}$-injectives and on some properties of the injective
hull in the category $\mathrm{H}^{\ast}V-\mathcal{U}$.
\medskip\\
Our work is naturally motivated by topology as shown in the study of
homotopy fixed points of a $\mathbb{Z}/2$-action (see \cite{L1}). Let $X$ be a
space equipped with an action of $\mathbb{Z}/2$ and
$X^{\mathrm{h}\mathbb{Z}/2}$ denote the space of homotopy fixed
points of this action. The problem of determining the
$\bmod{.\hspace{2pt}2}$ cohomology of $X^{\mathrm{h}\mathbb{Z}/2}$
(we ignore deliberately the questions of $2$-completion) involves
two steps:
\begin{itemize}
\item [--] determining the
$\bmod{.\hspace{2pt}2}$ equivariant cohomology
$\mathrm{H}^{*}_{\mathbb{Z}/2}X$;
\item [--] determining
$\mathrm{Fix}_{\mathbb{Z}/2} \hspace{2pt}
\mathrm{H}^{*}_{\mathbb{Z}/2}X$ (for the definition of the functor
$\mathrm{Fix}_{\mathbb{Z}/2}$ see section 2).
\end{itemize}
For the first step, see for example \cite{DL}, the main
information one has about the $\mathbb{Z}/2$-space $X$ is  that
the Serre spectral sequence, for $\bmod{.\hspace{2pt}2}$
cohomology, associated to the fibration
$$
X\rightarrow X_{\mathrm{h}\mathbb{Z}/2} \rightarrow\mathrm{B}
\mathbb{Z}/2
$$
collapses ($X_{\mathrm{h} \mathbb{Z}/2}$ denotes the Borel
construction $\mathrm{E}\mathbb{Z}/2 \times_{\mathbb{Z}/2}X$).
This collapsing implies that $\mathrm{H}^{*}_{\mathbb{Z}/2}X$ is
$\mathrm{H}$-free and that $\overline{
\mathrm{H}^{*}_{\mathbb{Z}/2}X}$ is canonically isomorphic to
$\mathrm{H}^{*}X$. This gives clearly a topological application of problem $( \mathcal{P})$. \\

We then prove the following results (related to the case $\overline{E}$ is $\sum^{n}\mathbb{F}_{2}$ and $ \mathrm{J}(2)$).
\begin{thm} Let $E$ be unstable $\mathrm{H}^{\ast}V-A$-module which is free as an
$\mathrm{H}^{\ast}V$-module.
If $\overline{E}$ is isomorphic to $\sum^{n}\mathbb{F}_{2}$, then there exists
an element $u$ in $\mathrm{H}^{\ast}V$ such that:
\begin{enumerate}
\item $u = \displaystyle\prod_{i} \theta_{i}^{\alpha_{i}}$, where $\theta_{i} \in
(\mathrm{H}^{1}V)\setminus \{0\}$ and $\alpha_{i} \in \mathbb{N}$
\item $E \cong \sum^{d}u \mathrm{H}^{\ast}V$ with $d+\displaystyle\sum_{i} \alpha_{i}=n$
\end{enumerate}
\end{thm}
\begin{prop} Let $E$ be an $ \mathrm{H}-A$-module which is $ \mathrm{H}$-free and such that $\overline{E}$
is isomorphic to $ \mathrm{J}(2)$ then:\\
$ E \cong \mathrm{H} \otimes  \mathrm{J}(2)$\\
or \\
$E$ is the sub-$\mathrm{H}-A$-module of $\mathrm{H} \oplus \sum \mathrm{H}$ generated by $(t, \Sigma1)$ and $(t^{2},0)$.
\end{prop}
The proofs of these two results are based on Smith theory, some properties of the functor $\mathrm{F}ix$
and on a result of J.P. Serre.
\bigskip\\
The paper is structured as follows. In section 2, we introduce the
definitions of reduced and nil-closed unstable $A$-modules. We give
the classification of injective modules in the category $\mathcal{U}$ and
in the category $H^{*}V-\mathcal{U}$. We also recall the algebraic Smith theory.
In section 3, we establish some
properties of $E$ when $\overline{E}$ is a reduced unstable $A$-module.
The results will be useful in section 4, where we give the solutions
of the problem ($\mathcal{P}$) when $E$ is free as an $\mathrm{H}^{\ast}V$-module
and $\overline{E}$ is nil-closed.
In section 5, we give some topological applications.
In section 6, we give the solutions
of the problem ($\mathcal{P}$) when $E$ is free as an $\mathrm{H}^{\ast}V$-module
and $\overline{E}$ is isomorphic to $\sum^{n}\mathbb{F}_{2}$, we also give a topological application.
In section 7, we solve the problem ($\mathcal{P}$) when
$\overline{E}$ is the Brown-Gitler module $\mathrm{J}(2)$ and $V$ is $\mathbb{Z}/2\mathbb{Z}$.

\textbf{Acknowledgements.} I would like to thank Professor Jean Lannes and Professor Said Zarati for several useful discussions. I am grateful to
the referee for his suggestions.

\section{Preliminaries on the categories $\mathcal{U}$ and
$\mathrm{H}^{\ast}V-\mathcal{U}$}
In this section, we will fix some notations, recall some definitions and
results about the categories $\mathcal{U}$ and $\mathrm{H}^{\ast}V-\mathcal{U}$.
\subsection{Nilpotent unstable $A$-modules}
Let $N$ be an unstable $A$-module. We denote by $Sq_{0}$ the
$\mathbb{Z}/2\mathbb{Z}$-linear map:
$$Sq_{0}: N \rightarrow N, \; x \mapsto Sq_{0}(x)=Sq^{\mid x \mid}x.$$
An unstable $A$-module $N$ is called nilpotent if:
\begin{center}
$ \forall \; x \in N, \; \exists\; n \in \mathbb{N}; \; Sq_{0}^{n}x=0.$
\end{center}
For example, finite unstable $A$-modules and suspension of unstable $A$-modules
are nilpotent. Let $Tor^{\mathrm{H}^{\ast}V}_{1}(\mathbb{F}_{2}, N)$ be the first
derived functor of the functor $\mathbb{F}_{2} \otimes_{\mathrm{H}^{\ast}V}- \; :
\mathrm{H}^{\ast}V-\mathcal{U} \rightarrow \mathcal{U}$, we have the following
useful result.
\begin{propo} (\cite{S} page 150)
Let $N$ be an unstable $\mathrm{H}^{\ast}V-A$-module, then the unstable $A$-module
$Tor^{\mathrm{H}^{\ast}V}_{1}(\mathbb{F}_{2}, N)$ is nilpotent.
\end{propo}
\subsection{Reduced unstable $A$-modules}
An unstable $A$-module $M$ is called reduced if the
$\mathbb{Z}/2\mathbb{Z}$-linear map:
$$Sq_{0}: M \rightarrow M, \; x \mapsto Sq_{0}(x)=Sq^{\mid x \mid}x,$$
is an injection.
\smallskip\\
Another characterization of reduced unstable $A$-module in terms of
nilpotent modules is the following.
\begin{lema} (\cite{LZ1}) An unstable $A$-module is reduced if it
does not contain a non-trivial nilpotent module.
\end{lema}
In particular, any $A$-linear map from a nilpotent $A$-module to a reduced
one is trivial.
\subsection{Nil-closed unstable $A$-modules}
Let $M$ be an unstable $A$-module. We denote by $Sq_{1}$ the
$\mathbb{Z}/2\mathbb{Z}$-linear map:
$$Sq_{1}: N \rightarrow N, \; x \mapsto Sq_{1}(x)=Sq^{\mid x \mid-1}x.$$
\begin{defnt} (\cite{EP}) An unstable $A$-module $M$ is called nil-closed if:
\begin{center}
\begin{enumerate}
    \item $M$ is reduced.
    \item $Ker(Sq_{1})=Im(Sq_{0})$.
\end{enumerate}
\end{center}
\end{defnt}
We have the following two characterizations of unstable nil-closed
$A$-modules.
\begin{lema} (\cite{LZ1})
Let $M$ be an unstable $A$-module and $\mathcal{E}(M)$ be its
injective hull. The unstable $A$-module $M$ is nil-closed if and only if
$M$ and the quotient $ \mathcal{E}(M)/ M $ are reduced.
\end{lema}
Let $Ext^{s}_{\mathcal{U}}(-, M)$ be the s-th derived functor of the
functor $\mathrm{H}om_{\mathcal{U}}(-, M)$.
\begin{lema} (\cite{LZ1}) An unstable $A$-module $M$ is nil-closed if and only if
$Ext^{s}_{\mathcal{U}}(N, M)=0$ for any nilpotent unstable $A$-module $N$
and $s=0,1$.
\end{lema}
\subsection{Injectives in the category $\mathcal{U}$}
Let $I$ be an unstable $A$-module, $I$ is called an injective in the category
$\mathcal{U}$ or $\mathcal{U}$-injective for short, if the functor
$\mathrm{H}om_{\mathcal{U}}(-, I)$ is exact.\\
The classification of $\mathcal{U}$-injectives (see \cite{LZ1},
\cite {LS}) is the following.\\
Let $\mathrm{J}(n), \; n \in \mathbb{N}$, be the $n$-th Brown- Gitler module,
characterized up to isomorphism, by the functorial bijection on the
unstable $A$-module M:
$$\mathrm{H}om_{\mathcal{U}}(M,\mathrm{J}(n)) \cong \mathrm{H}om_{\mathbb{F}_{2}}(M^{n},
\mathbb{F}_{2})$$
Clearly $\mathrm{J}(n)$ is an $\mathcal{U}$-injective and it is a finite module.\\
Let $\mathcal{L}$ be a set of representatives for $\mathcal{U}$-isomorphism
classes of indecomposable direct factors of
$\mathrm{H}^{\ast}(\mathbb{Z}/2\mathbb{Z})^{m},\; m \in \mathbb{N}$ (each class is
represented in $\mathcal{L}$ only once).\\ We have:
\begin{theo}
Let $I$ be an $\mathcal{U}$-injective module. Then there exists a set of
 cardinals $a_{L,n} \; , (L,n) \in \mathcal{L} \times
\mathbb{N}$, such
that $ I \cong \displaystyle \bigoplus_{(L,n)} (L \otimes \mathrm{J}(n))^{ \oplus a_{L,n}} $ .\\
Conversely, any
unstable $A$-module of that form is $\mathcal{U}$-injective.
\end{theo}
Let's remark that $\mathrm{H}^{\ast}V$ is an $\mathcal{U}$-injective.
\subsection{The injectives of the category $\mathrm{H}^{\ast}V-\mathcal{U}$}
The classification of injectives of the category $\mathrm{H}^{\ast}V-\mathcal{U}\;
\; (\mathrm{H}^{\ast}V-\mathcal{U}$-injectives
for short) is given by Lannes-Zarati \cite{LZ2} as follows.\\
Let $\mathrm{J}_{V}(n), \; n \in \mathbb{N}$, be the unstable $\mathrm{H}^{\ast}V-A$-module
characterized, up to isomorphism, by the functorial bijection on the
unstable $\mathrm{H}^{\ast}V-A$-module M:
$$\mathrm{H}om_{\mathrm{H}^{\ast}V-\mathcal{U}}(M,\mathrm{J}_{V}(n)) \cong
\mathrm{H}om_{\mathbb{F}_{2}}(M^{n},\mathbb{F}_{2})$$
Clearly $\mathrm{J}_{V}(n)$ is an $\mathrm{H}^{\ast}V-\mathcal{U}$-injective.\\
Let $\mathcal{W}$ be the set of subgroups of $V$ and let $(W,n) \in \mathcal{W}
\times \mathbb{N}$, we write $$E(V,W,n) =
\mathrm{H}^{\ast}V \otimes_{\mathrm{H}^{\ast}V/W} \mathrm{J}_{V/W}(n)$$
(in this formula $\mathrm{H}^{\ast}V$
is an $\mathrm{H}^{\ast}V/W$-module via the map induced in mod.2 cohomology by the
canonical projection $ V \rightarrow V/W $).\\
\begin{theo}(\cite{LZ2}) If I is an injective of the category of
$\mathrm{H}^{\ast}V-\mathcal{U}$, then $I \cong \displaystyle \bigoplus_{(L,W,n)
\in \mathcal{L} \times \mathcal{W} \times \mathbb{N}} (E(V,W,n)
\otimes_{\mathbb{F}_{2}} L)^{\oplus_{a_{L,W,n}}}$.
\smallskip\\
Conversely, each
$\mathrm{H}^{\ast}V-A$-module of this form is an $\mathrm{H}^{\ast}V-\mathcal{U}$-injective.
\end{theo}
Clearly $\mathrm{H}^{\ast}V$ is an $\mathrm{H}^{\ast}V-\mathcal{U}$-injective.
\subsection{Algebraic Smith theory}
\subsubsection{The functors $\mathrm{F}ix$}
We introduce the functors $\mathrm{F}ix$ (\cite{L1}, \cite{LZ2}). We denote by
$$\mathrm{F}ix_{V}: \mathrm{H}^{\ast}V-\mathcal{U} \rightarrow \mathcal{U}$$
the left adjoint of the functor
$$\mathrm{H}^{\ast}V \otimes-: \mathcal{U} \rightarrow
\mathrm{H}^{\ast}V-\mathcal{U}$$
We have the functorial bijection:
\begin{center}
$\displaystyle \mathrm{H}om_{\mathrm{H}^{\ast}V-\mathcal{U}}(N,\;H^{\ast}V \otimes P) \cong
\mathrm{H}om_{\mathcal{U}}(\mathrm{F}ix_{V}N,\;P)$
\end{center}
for every unstable $\mathrm{H}^{\ast}V-A$-module $N$ and every unstable $A$-module $P$.\\
The functor $\mathrm{F}ix_{V}$ has the following properties.\\
2.6.1.1. The functor $\mathrm{F}ix_{V}$ is an exact functor.\\
2.6.1.2. Let $N$ be an unstable $ \mathrm{H}^{\ast}V-A$-module and $\mathcal{E}(N)$ be its
injective hull. Then, the module $\mathrm{F}ix_{V}\mathcal{E}(N)$
is the injective hull of $\mathrm{F}ix_{V}N$.
\subsubsection{} Let $N$ be an unstable $\mathrm{H}^{\ast}V-A$-module, we denote by
$$\eta_{_{V}}:\; N \rightarrow \mathrm{H}^{\ast}V \otimes \mathrm{F}ix_{_{V}}N$$
the adjoint of the identity of
$Fix_{_{V}}N$.
We denote by $\mathrm{c}_{V}=\displaystyle \prod_{u \in \mathrm{H}^{1}V -\{0\}} u$ the top Dickson invariant, we have
the following result (see \cite{LZ2} corollary 2.3).
\begin{propo} Let $N$ be an unstable $\mathrm{H}^{\ast}V-A$-module. The
localization of the map $\eta_{_{V}}$
$$ \eta_{_{V}} [\mathrm{c}_{V}^{-1}]: N[\mathrm{c}_{V}^{-1}] \rightarrow \mathrm{H}^{\ast}V
[\mathrm{c}_{v}^{-1}] \otimes \mathrm{F}ix_{V}N $$
is an injection.
\end{propo}

This shows in particular, that if $N$ is torsion-free then the map $\eta_{_{V}}$
is an injection.

The proposition 2.6.1 can be reformulated as follows.
\begin{propo} Let $N$ be an unstable $\mathrm{H}^{\ast}V-A$-module. If
$N$ is torsion-free then its injective hull in
$\mathrm{H}^{\ast}V-\mathcal{U}$ is free as an $H^{*}V$-module and is isomorphic
to $\displaystyle
\bigoplus_{(L,n) \in \mathcal{L} \times \mathbb{N}}
(\mathrm{H}^{\ast}V \otimes \mathrm{J}(n))\otimes L$
\end{propo}
\begin{proof} Since the module is torsion-free then the
map $\eta_{_{V}}:  \; N \rightarrow \mathrm{H}^{\ast}V \otimes \mathrm{F}ix_{_{V}}N$
adjoint of the identity of $\mathrm{F}ix_{_{V}}N$ is an injection. So $N$ is a
sub-$\mathrm{H}^{\ast}V-A$-module of $\mathrm{H}^{\ast}V \otimes \mathrm{F}ix_{_{V}}N$. By 2.6.1.1
and 2.6.1.2, we have that
the injective hull of $N$ is isomorphic
to $\mathrm{H}^{\ast}V \otimes I$, where $I$ is an $\mathcal{U}$-injective.
\end{proof}
\begin{remk} { \em As a consequence of proposition 2.6.2, we have that
if $E$ is an unstable $\mathrm{H}^{\ast}V-A$-module which is
free as an $\mathrm{H}^{\ast}V$-module then
its injective hull (in the category $\mathrm{H}^{\ast}V-\mathcal{U}$) is also free as
an $\mathrm{H}^{\ast}V$-module.}
\end{remk}
\begin{propo} \cite{LZ2}. Let $N$ be an unstable $\mathrm{H}^{\ast}V-A$-module which is of finite
type as an $\mathrm{H}^{\ast}V$-module. The
localization of the map $\eta_{_{V}}$
$$ \eta_{_{V}} [\mathrm{c}_{V}^{-1}]: N[\mathrm{c}_{V}^{-1}] \rightarrow \mathrm{H}^{\ast}V
[\mathrm{c}_{V}^{-1}] \otimes \mathrm{F}ix_{V}N $$
is an isomorphism.
\end{propo}
In particular, the previous result shows that:
\begin{enumerate}
\item If $N$ is free as an $\mathrm{H}^{\ast}V$-module, then the map $\eta_{V}$ is an injection.
\item The isomorphism of the proposition proves that $dim \overline{E}= dim \mathrm{F}ix_{V}E$ where $dim$ is the
total dimension (see \cite{LZ2}).
\end{enumerate}
\section{Some properties of $E$ when $\overline{E}$ is reduced}
In this section we will prove some algebraic results which will be useful
for section 4. In fact, we will analyze the relation between an unstable
$\mathrm{H}^{\ast}V-A$-module $E$ and its (associated) unstable $A$-module
$\overline{E}$.
For this, we will begin by giving some technical results.
\subsection{Technical results}
\begin{lema}
Let $P$ and $Q$ be unstable $\mathrm{H}^{\ast}V-A$-modules, free as
$\mathrm{H}^{\ast}V$-modules
and $f: P \rightarrow Q$ an $\mathrm{H}^{\ast}V-A$-linear map. If the induced map
$ \overline{f}: \overline{P} \rightarrow \overline{Q}$ is an injection
then $f$ is also an injection.
\end{lema}
\begin{proof}
Let's denote by $Imf$ the image of $f$, by $\widetilde{f} : P \rightarrow Imf$
the natural surjection and by
$i: Imf \hookrightarrow Q$ the inclusion of $Imf$ in $Q$. Since
$ \overline{f}$ is an injection so the induced map $ \overline{(\widetilde{f})}$
is an isomorphism of unstable $A$-modules and then the induced map
$ \overline{i}$ is an injection. This shows that $ \overline{Imf}$ is the image of
$ \overline{f}$. Since the module $Imf$ is a sub-$\mathrm{H}^{\ast}V$-module of
the $\mathrm{H}^{\ast}V$-free module $Q$ and
$\overline{i}: \overline{Imf} \hookrightarrow \overline{Q}$ is an injection, so
$Imf$ is free as an $\mathrm{H}^{\ast}V$-module.
In particular, we have that $Tor_{1}^{\mathrm{H}^{\ast}V}(\mathbb{F}_{2}, Imf)$=0 (see for example \cite {R}).
Let's denote by $N$ the kernel of the map $\widetilde{f}$, so we have the following short exact
sequence in $\mathrm{H}^{\ast}V-\mathcal{U}$:
$$
    \xymatrix{0 \ar[r] & N \ar[r]^-{} &
P \ar[r]^-{\widetilde{f}} & Imf \ar[r] & 0}\,.
$$
By applying the functor ($ \mathbb{F}_{2} \otimes_{\mathrm{H}^{\ast}V}-$) to the
previous sequence, we prove that $\overline{N}$ is trivial (since the map
$ \overline{(\widetilde{f})}$ is an isomorphism and $Imf$ is free as an $\mathrm{H}^{\ast}V-A$-module).
Hence the module $N$ is trivial and the map $f$ is an injection.

\end{proof}
The converse of this lemma is not true in general, but we have the following result:
\begin{lema}
Let $P$ and $Q$ be unstable $\mathrm{H}^{\ast}V-A$-modules, free as $\mathrm{H}^{\ast}V$-modules
and $f: P \rightarrow Q$ an $\mathrm{H}^{\ast}V-A$-linear map which is an injection. If
$\overline{P}$ is a reduced unstable $A$-module, then the induced map
$ \overline{f}: \overline{P} \rightarrow \overline{Q}$ is an injection.
\end{lema}
\begin{proof}
We denote by $C$ the quotient of $Q$ by $P$, we have the following short exact
sequence in $\mathrm{H}^{\ast}V-\mathcal{U}$:
$$
    \xymatrix{0 \ar[r] & P \ar[r]^-{f} &
Q \ar[r]^-{} & C \ar[r] & 0}\,.
$$
By applying the functor ($ \mathbb{F}_{2} \otimes_{\mathrm{H}^{\ast}V}-$) to the
previous sequence, we obtain an exact sequence in $ \mathcal{U}$:
$$
    \xymatrix{0 \ar[r] &  Tor^{\mathrm{H}^{\ast}V}_{1}(\mathbb{F}_{2}, C) \ar[r] & \overline{P}
     \ar[r]^-{\overline{f}} & \overline{Q} \ar[r]^-{} & \overline{C} \ar[r] & 0}\,.
$$
Since $\overline{P}$ is reduced as unstable $A$-module and
$Tor_{1}^{{\mathrm{H}^{\ast}V}}(\mathbb{F}_{2}, C)$ is nilpotent (see proposition
2.1.1), then the map $ \overline{f}$ is  an injection.

\end{proof}
\subsection{Statement of some properties of $E$ when $\overline{E}$ is
reduced }
The first result of this paragraph concerns the relation between the injective
hull of $E$ and the induced module $ \overline{E}$.
\begin{theo} Let $E$ be an unstable $\mathrm{H}^{\ast}V-A$-module which is free as an
$H^{*}V$-module and let
$\mathcal{E}(E)$ be its injective hull (in the category
$\mathrm{H}^{\ast}V-\mathcal{U}$). We suppose that $\overline{E}$ is reduced and let
$I$ be its injective hull in the category $\mathcal{U}$.\\
Then $\mathcal{E}(E)$ is isomorphic, as an unstable $\mathrm{H}^{*}V-A$-module,
to $\mathrm{H}^{*}V \otimes I$.
\end{theo}
\begin{proof} Since $E$ is free as an $\mathrm{H}^{*}V$-module, then
$ \mathcal{E}(E)$ is isomorphic,
in the category $\mathrm{H}^{*}V-\mathcal{U}$, to $\mathrm{H}^{*}V \otimes J$,
where $J$ is an $ \mathcal{U}$-injective
(see proposition 2.6.2).\\
Let's denote by $i$ the inclusion of $E$ in
$\mathcal{E}(E)$, we have, by lemma 3.1.2, that the induced map $ \overline{i}$ is
 an injection. We will
prove, by
using the definition, that $J$ is the
injective hull of $\overline{E}$, in the category $\mathcal{U}$.
Let $P$ be a sub-$A$-module of
$J$ such that the $A$-module
$(\overline{i})^{-1}(P)$ is trivial, we have to show that the unstable
$A$-module $P$ is trivial.\\
Since $(\overline{i})^{-1}(P)$ is trivial then the composition:
$\pi \circ \overline{i}: \xymatrix{ \overline{E} \ar[r]^-{ \overline{i}} &
J \ar[r]^-{ \pi} & J/P }$
is  an injection. By lemma 3.1.1, the following composition\\
$\xymatrix{ E \ar[r]^-{ i} &
\mathrm{H}^{*}V \otimes J \ar[r]^-{} & \mathrm{H}^{*}V \otimes (J/P) }$
is  an injection, which proves that the unstable $\mathrm{H}^{*}V-A$-module
$i^{-1}(\mathrm{H}^{*}V \otimes P)$ is trivial.
Since $\mathrm{H}^{*}V \otimes J$ is the injective hull of $E$ so the unstable
$\mathrm{H}^{*}V-A$-module $\mathrm{H}^{*}V \otimes P$ is trivial.
\end{proof}
\begin{coro}
Let $E$ be an unstable $\mathrm{H}^{\ast}V-A$-module such that:
\begin{enumerate}
    \item  $E$ is free as an $\mathrm{H}^{*}V$-module.
    \item  $\overline{E}$ is reduced as unstable $A$-module.
\end{enumerate}
Then $E$ is reduced as unstable $A$-module.
\end{coro}
\begin{proof} We have, by theorem 3.2.1, that the injective hull of $E$ is
$\mathrm{H}^{*}V \otimes I$, where
$I$ is the injective hull of $ \overline{E}$ in $ \mathcal{U}$. Since
$\overline{E}$ is reduced, then $I$ is a reduced $ \mathcal{U}$-injective.
This shows that $E$ is reduced as an unstable $A$-module because its injective
hull (in the category $\mathrm{H}^{\ast}V-\mathcal{U}$) is
$\mathrm{H}^{*}V \otimes I$ which is reduced as unstable $A$-module.
\end{proof}
\begin{remk} {\em In the previous result the condition (1): $E$ is free as an
$\mathrm{H}^{*}V$-module is necessary. In fact, the finite $\mathrm{H}-A$-module
$\mathrm{J}_{\mathbb{Z}/2\mathbb{Z}}(1)$ is not free as an $\mathrm{H}$-module and not
reduced as an unstable $A$-module, however
$\overline{\mathrm{J}_{\mathbb{Z}/2\mathbb{Z}}(1)}= \mathbb{F}_{2}$ is
a reduced unstable $A$-module. Observe that $\mathrm{J}_{\mathbb{Z}/2\mathbb{Z}}(1)$
is isomorphic, as unstable $A$-module, to
$\mathbb{F}_{2} \oplus \sum \mathbb{F}_{2}$, the structure of
$\mathrm{H}$-module is given by: $t.\iota=\Sigma \iota$, where $\iota$ is
the generator of $\mathbb{F}_{2}$ and $t$ the generator of $\mathrm{H}$.\\
Observe that the converse of corollary 3.2.2 is false. In fact, the $\mathrm{H}-A$-module $E= \mathrm{H}^{\geq1}$ is reduced as
unstable $A$-module however the unstable $A$-module $\overline{E}\cong \sum\mathbb{F}_{2}$ is not reduced.}
\end{remk}
\section{Description of $E$ when $\overline{E}$ is nil-closed}
The main result of this paragraph concerns the relation between the two first
terms of a (minimal) injective resolution of $E$ and $ \overline{E}$.
\begin{thm} Let $E$ be an unstable $\mathrm{H}^{\ast}V-A$-module which is free as an
$\mathrm{H}^{*}V$-module. We suppose that:
\begin{enumerate}
     \item $\overline{E}$ is nil-closed.
     \item $\xymatrix{0 \ar[r] & \overline{E} \ar[r] &
  I_{0} \ar[r]^{i_{1}} & I_{1} \ar[r] & ....} $ is the beginning of a (minimal) $\mathcal{U}$-
  injective resolution of $\overline{E}$.
\end{enumerate}
Then there exists an $ \mathrm{H}^{\ast}V-A$-linear map $\varphi: \mathrm{H}^{*}V \otimes I_{0} \rightarrow
\mathrm{H}^{*}V \otimes I_{1}$ such that:
\begin{enumerate}
\item $\xymatrix{0 \ar[r] & E \ar[r] &
  \mathrm{H}^{*}V \otimes I_{0} \ar[r]^{\varphi} & \mathrm{H}^{*}V \otimes I_{1}
  \ar[r] & ....} $ is the beginning of
 a (minimal) injective resolution of $E$ (in the category
$\mathrm{H}^{*}V-\mathcal{U}$).
\item $ \overline{\varphi}= i_{1} $
\end{enumerate}
\end{thm}
\begin{proof} The unstable $A$-module $ \overline{E}$ is nil-closed so is reduced,
we have then, by theorem 3.2.1, that the injective hull of $E$ is
$\mathrm{H}^{*}V \otimes I_{0}$.
We denote by $C_{0}$ the quotient of $\mathrm{H}^{\ast}V \otimes I_{0}$ by $E$.
We have the following short exact sequence in $\mathrm{H}^{\ast}V-\mathcal{U}$:
$$
    \xymatrix{0 \ar[r] & E \ar[r]^-{ i_{0}} &
  \mathrm{H}^{*}V \otimes I_{0} \ar[r]^-{} & C_{0} \ar[r] & 0}\,.
$$
Since the induced map $\overline{i_{0}}$ is  an injection (see lemma 3.1.2), then
the unstable $A$-module $Tor_{1}^{\mathrm{H}^{\ast}V}(\mathbb{F}_{2}, C_{0})$ is trivial;
this shows that the module $C_{0}$ is free as an $\mathrm{H}^{*}V$-module (see for example
\cite{NS}, proposition A.1.5).\\
We verify that the $\mathcal{U}$-injective hull of $\overline{C_{0}}$ is $I_{1}$
and that $C_{0}$ is reduced since $\overline{C_{0}}$ is reduced (see corollary 3.2.2).
This implies, by theorem 3.2.1, that the $\mathrm{H}^{*}V-\mathcal{U}$-injective
hull of $C_{0}$ is isomorphic to $\mathrm{H}^{*}V \otimes I_{1}$.
\end{proof}
\begin{rem} {\em let $M$ be a nil-closed unstable $A$-module and \\
$\xymatrix{0 \ar[r] & M \ar[r]^-{i_{0}} & I_{0} \ar[r]^-{i_{1}}
& I_{1} \ar[r] & ....} $ be the beginning of a (minimal) $\mathcal{U}$-injective
resolution of $M$. We denote by
$$(\mathrm{H}om_{ \mathrm{H}^{\ast}V-\mathcal{U}}(\mathrm{H}^{\ast}V \otimes I_{0},\;
\mathrm{H}^{\ast}V \otimes I_{1}))_{i_{1}}$$
the set of $\mathrm{H}^{\ast}V-A$-linear map
$\varphi: \mathrm{H}^{\ast}V \otimes I_{0} \rightarrow \mathrm{H}^{\ast}V \otimes I_{1}$ such that
$\overline{\varphi}= i_{1}$.\\
Using Lannes T-functor (see \cite{L1}) we have:
$$(\mathrm{H}om_{ \mathrm{H}^{\ast}V-\mathcal{U}}(\mathrm{H}^{\ast}V \otimes I_{0},\;
\mathrm{H}^{\ast}V \otimes I_{1}))_{i_{1}} \cong
(\mathrm{H}om_{\mathcal{U}}(T_{V}I_{0},\; I_{1}))_{i_{1}}$$
where $(\mathrm{H}om_{\mathcal{U}}(T_{V}I_{0},\;I_{1}))_{i_{1}}$ is the set of
$A$-linear map $\psi: T_{V}I_{0} \rightarrow I_{1}$ such that
$\psi \circ i=i_{1}$, where $i: I_{0} \hookrightarrow T_{V}I_{0}$ denotes the
natural inclusion.\\
The kernel of any element $\psi \in (\mathrm{H}om_{\mathcal{U}}(T_{V}I_{0},\;
I_{1}))_{i_{1}}$, which is free as an $\mathrm{H}^{*}V$-module, is an unstable
$\mathrm{H}^{\ast}V-A$-module such that $\overline{ker \psi} \cong M$.}
\end{rem}
\begin{rem} {\em If $ \overline{E}$ is an $ \mathcal{U}$-injective then the only
unstable free $\mathrm{H}^{\ast}V-A$-module, up to isomorphism, solution of
the problem $(\mathcal{P})$ is $\mathrm{H}^{\ast}V \otimes \overline{E}$.\\
Let $n$ be an even integer. The unstable free
$\mathrm{H}-A$-modules, up to isomorphism, solution of the problem
$(\mathcal{P})$ when $M$ is $\mathrm{H}^{*}BSO(n)$ are $\mathrm{H}^{*}BO(n)$ and
$\mathrm{H} \otimes \mathrm{H}^{*}BSO(n)$. We verify that these two
$\mathrm{H}-A$-modules are not isomorphic in the category
$\mathrm{H}- \mathcal{U}$ (since it does not exist an $A$-linear section of the
projection $\mathrm{H}^{*}BO(n) \rightarrow \mathrm{H}^{*}BSO(n)$).}
\end{rem}
\section{Applications}
\subsection{} Our first application
concerns the determination of the $\bmod{.\hspace{2pt}2}$
cohomology of the mapping space $\mathbf{hom}
\hspace{1pt}(\mathrm{B}\hspace{1pt} (\mathbb{Z}/2^{n}),Y)$ whose
domain is a classi\-fying space for the group $\mathbb{Z}/2^{n}$
and whose range is a space $Y$ such that $\mathrm{H}^{*}Y$ is
concentrated in even degrees.

\medskip
We will just recall some facts, ignoring the p-completion
problems. For further details see \cite{DL}.

\medskip
One proceeds by induction on the integer $n$. Let us set
$$
\hspace{24pt} X=\mathbf{hom} \hspace{1pt}(\mathrm{E}\hspace{1pt}
(\mathbb{Z}/2^{n})/ (\mathbb{Z}/2^{n-1}),Y) \hspace{24pt}.
$$
The space $X$ has the homotopy type of $\mathbf{hom}
\hspace{1pt}(\mathrm{B}\hspace{1pt} (\mathbb{Z}/2^{n-1}),Y)$ and
is equipped of an action $\mathbb{Z}/2$ such that one has a
homotopy equivalence
$$
\hspace{24pt} \mathbf{hom} \hspace{1pt}(\mathrm{B}\hspace{1pt}
(\mathbb{Z}/2^{n}),Y)\cong X^{\mathrm{h}\hspace{1pt} \mathbb{Z}/2}
\hspace{24pt},
$$
$X^{\mathrm{h}\hspace{1pt} \mathbb{Z}/2}$ denoting the homotopy
fixed point space: $\mathbf{hom}_{\mathbb{Z}/2}
\hspace{1pt}(\mathrm{E}\hspace{1pt} \mathbb{Z}/2,X)$. Using
$\mathrm{Fix}_{\mathbb{Z}/2}$-theory \cite{L1}, one gets:
$$
\hspace{24pt} \mathrm{H}^{*} \mathbf{hom}
\hspace{1pt}(\mathrm{B}\hspace{1pt} (\mathbb{Z}/2^{n}),Y) \cong
\mathrm{Fix}_{\mathbb{Z}/2} \hspace{2pt}
\mathrm{H}^{*}_{\mathbb{Z}/2} \hspace{1pt}X \hspace{24pt}.
$$
Since the computation of the functor $\mathrm{Fix}_{\mathbb{Z}/2}$
on an unstable $\mathrm{H}-\mathrm{A}$-module is not difficult in
general, the determination of the $\bmod{.\hspace{2pt}2}$
cohomology of the mapping space $\mathbf{hom}
\hspace{1pt}(\mathrm{B}\hspace{1pt} (\mathbb{Z}/2^{n}),Y)$ is
reduced to the determination of the unstable
$\mathrm{H}-\mathrm{A}$-module $\mathrm{H}^{*}_{\mathbb{Z}/2}
\hspace{1pt}X$. As we are going to explain, this last point is
closely related to problem $(\mathcal{P})$.

\medskip
One knows by induction on $n$ that the $\bmod{.\hspace{2pt}2}$
cohomology of the space $X$ as the one of the space $Y$ is
concentrated in even degrees and one checks that the action of
$\mathbb{Z}/2$ on $\mathrm{H}^{*}(Y;\mathbb{Z})$ is trivial. These
two facts imply that the Serre spectral sequence, for
$\bmod{.\hspace{2pt}2}$ cohomology, associated to the fibration
$$
X\rightarrow X_{\mathrm{h}\mathbb{Z}/2} \rightarrow\mathrm{B}
\mathbb{Z}/2
$$
collapses ($X_{\mathrm{h} \mathbb{Z}/2}$ denotes the Borel
construction $\mathrm{E}\mathbb{Z}/2 \times_{\mathbb{Z}/2}X$).
This collapsing implies in turn that
$\mathrm{H}^{*}_{\mathbb{Z}/2}X$ is $\mathrm{H}$-free and that
$\overline{ \mathrm{H}^{*}_{\mathbb{Z}/2}X}$ is isomorphic to
$\mathrm{H}^{*}X$. So the determination of $\mathrm{H}^{*}
\mathbf{hom} \hspace{1pt}(\mathrm{B}\hspace{1pt}
(\mathbb{Z}/2^{n}),Y)$ is indeed reduced to the resolution of a
problem $(\mathcal{P})$.

\medskip
We conclude this subsection by a concrete example (we follow
\cite{De}, section 6); we take $n=2$ and $Y=\mathrm{BSU}(2)$.
Using $\mathrm{T} _{\mathbb{Z}/2}$-computations one sees that $X$
has the homotopy type of $\mathrm{BSU}(2)\coprod \mathrm{BSU}(2)$;
one checks also that the $\mathbb{Z}/2$-action preserves the
connected components. The $(\mathcal{P})$-problem asociated to the
determination of the unstable $\mathrm{H}-\mathrm{A}$-module
$\mathrm{H}^{*}_{\mathbb{Z}/2} \hspace{1pt}X$ is the following
one:

\smallskip
Find the unstable $\mathrm{H}-\mathrm{A}$-modules $E$ such that
\begin{itemize}
\item [--] $E$ is
$\mathrm{H}$-free;
\item [--] the unstable
$\mathrm{A}$-module $\overline{E}$ is isomorphic to
$\mathrm{H}^{*}\mathrm{BSU}(2)$.
\end{itemize}
Using the fact that the injective hull, in the category
$\mathrm{H}-\mathcal{U}$, of $E$ is $\mathrm{H}\otimes\mathrm{H}$
(see theorem 3.2), one checks that one has two possibilities:
\begin{itemize}
\item [--]
$E\cong\mathrm{H}\otimes \mathrm{H}^{*}\mathrm{BSU}(2)$;
\item [--]
$E\cong\mathrm{H}\otimes _{\mathrm{H}^{*}\mathrm{BU}(1)}
\mathrm{H}^{*}\mathrm{BU}(2)$ (the structures of unstable
$\mathrm{H}^{*}\mathrm{BU}(1)- \mathrm{A}$-modules on
$\mathrm{H}=\mathrm{H}^{*} \mathrm{BO}(1)$ and
$\mathrm{H}^{*}\mathrm{BU}(2)$ are respectively induced by the
inclusion of $\mathrm{O}(1)$ in $\mathrm{U}(1)$ and the
determinant homomorphism from $\mathrm{U}(2)$ to $\mathrm{U}(1)$).
\end{itemize}
\medskip

\subsection{} The theorem 4.1 can be illustrated, topologically, as follows:
\begin{propo} Let $X$ be a CW-complex on which acts an elementary abelian
group 2-group $V$. Suppose that:
\begin{enumerate}
    \item $\mathrm{H}^{\ast}X$ is nil-closed
    \item $\xymatrix{0 \ar[r] &
    \mathrm{H}^{\ast}X \ar[r] &
  I_{0} \ar[r]^-{\alpha} & I_{1} \ar[r] & ....} $ is the beginning of a (minimal)
  $\mathcal{U}$-injective resolution of $\mathrm{H}^{\ast}X$
    \item $\mathrm{H}^{\ast}_{V}X$ is free as an $\mathrm{H}^{*}V$-module.
\end{enumerate}
Then there exists an $\mathrm{H}^{\ast}V-A$-linear map $\varphi: \;
\mathrm{H}^{*}V \otimes I_{0}
\rightarrow \mathrm{H}^{*}V \otimes I_{1}$ such that:
\begin{enumerate}
    \item $\mathrm{H}^{\ast}_{V}X \cong Ker(\varphi)$.
    \item $\xymatrix{0 \ar[r] & \mathrm{H}^{\ast}_{V}X \ar[r] &
  \mathrm{H}^{*}V \otimes I_{0} \ar[r]^-{\varphi} & \mathrm{H}^{*}V \otimes I_{1} \ar[r] & ....} $
  is the beginning of a (minimal) injective resolution of $\mathrm{H}^{\ast}_{V}X$ (in the category
  $\mathrm{H}^{*}V-\mathcal{U}$).
    \item $\overline{\varphi}= \alpha : I_{0} \rightarrow I_{1}$.
\end{enumerate}
\end{propo}
In particular, we have:
\begin{coro} Let $X$ be a CW-complex on which acts an elementary abelian
group 2-group $V$. Suppose that:
\begin{enumerate}
    \item $\mathrm{H}^{\ast}X$ is a reduced $\mathcal{U}$-injective,
    \item $\mathrm{H}^{\ast}_{V}X$ is free as an $\mathrm{H}^{*}V$-module.
\end{enumerate}
Then $\mathrm{H}^{\ast}_{V}X \cong \mathrm{H}^{\ast}V \otimes H^{*}X$.
\end{coro}

\section{Description of $E$ when $\overline{E}$ is isomorphic to $\sum^{n}\mathbb{F}_{2}$}
In this section, we prove the following result.
\begin{thm} Let $E$ be unstable $\mathrm{H}^{\ast}V-A$-module which is free as an
$\mathrm{H}^{\ast}V$-module.
If $\overline{E}$ is isomorphic to $\sum^{n}\mathbb{F}_{2}$, then there exists
an element $u$ in $\mathrm{H}^{\ast}V$ such that:
\begin{enumerate}
\item $u = \displaystyle\prod_{i} \theta_{i}^{\alpha_{i}}$, where $\theta_{i} \in
(\mathrm{H}^{1}V)\setminus \{0\}$ and $\alpha_{i} \in \mathbb{N}$
\item $E \cong \sum^{d}u \mathrm{H}^{\ast}V$ with $d+\displaystyle\sum_{i} \alpha_{i}=n$.
\end{enumerate}
\end{thm}
\begin{proof} Let $N$ be an unstable $A$-module, we denote by $\mathrm{d}imN$ the total dimension of $N$ that is
$\mathrm{d}im \; N= \sum_{i} \mathrm{d}im \; N^{i}$. We have the equality
$\mathrm{d}im \; \overline{E}=1 = \mathrm{d}im \; \mathrm{F}ix_{_{V}}E$ (see \cite{LZ3}), so we deduce that $\mathrm{F}ix_{_{V}}E=
\sum^{l}\mathbb{F}_{2}$, where
$l \in \mathbb{N}$. Let $\eta_{_{V}}:\; E \rightarrow \mathrm{H}^{\ast}V \otimes \mathrm{F}ix_{_{V}}E$
be the adjoint of the identity of $\mathrm{F}ix_{_{V}}E$ (see \cite{LZ2}). Since the map $\eta_{V}$ is an injection, then the module $E$ is a
sub-$\mathrm{H}^{\ast}V-A$-module of $\sum^{l}\mathrm{H}^{\ast}V$. Let's write $E=\sum^{l}E'$, where $E'$ is
sub-$\mathrm{H}^{\ast}V-A$-module of $\mathrm{H}^{\ast}V$ . By a result of J-P. Serre (see \cite{Se}),
there exists $N$ such that: $\mathrm{c}_{V}^{N} \mathrm{H}^{\ast}V \subset E' \subset \mathrm{H}^{\ast}V$. Since $E'$ is
free as an $\mathrm{H}^{\ast}V$-module and of dimension one, then there exists $u \in \mathrm{\widetilde{H}}^{\ast}V$
such that $E'= u \mathrm{H}^{\ast}V$. The inclusion $\mathrm{c}_{V}^{N} \mathrm{H}^{\ast}V \subset u \mathrm{H}^{\ast}V$
proves that $u= \displaystyle\prod_{i} \theta_{i}^{\alpha_{i}}$, where $\theta_{i} \in
(\mathrm{H}^{1}V)\setminus \{0\}$ and $\alpha_{i} \in \mathbb{N}$.
\end{proof}

\begin{rem} {\em We remark that by the previous result, we can determinate $E$ when $\overline{E}$ is isomorphic to
$\mathbb{F}_{2}\oplus
\sum^{n}\mathbb{F}_{2}$. In this case, we verify that $ E \cong \mathrm{H}^{\ast}V \oplus
\sum^{d}u \mathrm{H}^{\ast}V$, where $u = \displaystyle\prod_{i} \theta_{i}^{\alpha_{i}}$, $\theta_{i} \in
\mathrm{H}^{\ast}V\setminus \{0\}$, $\alpha_{i} \in \mathbb{N}$ and $d+\displaystyle\sum_{i} \alpha_{i}=n$.
In fact, since the $\mathrm{H}^{\ast}V-\mathcal{U}$-injective module $\mathrm{H}^{\ast}V$ is a sub-$\mathrm{H}^{\ast}V$-module of $E$,
then $E \cong \mathrm{H}^{\ast}V \oplus E'$, where $E'$ is
an unstable $\mathrm{H}^{\ast}V-A$-module, free as an
$\mathrm{H}^{\ast}V$-module and such that $ \overline{E'} \cong \sum^{n}\mathbb{F}_{2}$. The result holds from
theorem 6.1.}
\end{rem}
\textbf{6.3 \hspace{5mm} Example}\\
We give an example showing how to realize topologically the cases of theorem 6.1 and remark 6.2.\\
Let $\rho: V \rightarrow \mathrm{O}(d)$ be a group homomorphism. $\rho$ gives both an action of $V$ on $\mathrm{D}^{d}$, $\mathrm{S}^{d-1}$
and a $d$-dimensional orthogonal bundle whose mod.2 Euler class is denoted by $e(\rho)$.\\
The long exact sequence of the pair ($\mathrm{D}^{d},\; \mathrm{S}^{d-1}$) and the Thom isomorphism give the long (Gysin) exact sequence
(see for example \cite{Hu}):
$$
    \xymatrix{ \cdots  \ar[r] & \mathrm{H}^{*-1}V
    \ar[r]^-{}& \mathrm{H}_{V}^{*-1}\mathrm{S}^{d-1} \ar[r]^-{ } & \Sigma^{-d}\mathrm{H}^{*}V  \ar[r]^-{\smile e(\rho) } &  \mathrm{H}^{*}V \ar[r]^-{} &
    \mathrm{H}^{*}_{V}\mathrm{S}^{d-1}  \ar[r]^-{} & \cdots }
$$
The decomposition $ \rho \cong \displaystyle\oplus_{i=1}^{d} \; \rho_{i}$ of the representation $\rho$ into orthogonal representations of dimension 1 gives
$e(\rho)= \prod_{i}e(\rho _{i}).$ We have now two cases.\\
- If none of the representations $\rho _{i}$ is trivial then $e(\rho)$ is non zero and $ \mathrm{H}^{\ast}_{V}(\mathrm{D}^{d}, \mathrm{S}^{d-1})$ is
isomorphic to $ e(\rho) \mathrm{H}^{*}V$ as an $\mathrm{H}^{*}V-\mathrm{A}$-module. This illustrates theorem 6.1.\\
- Otherwise, let's write $\rho= \sigma \oplus \tau$, $\sigma$ (resp. $\tau$) being the direct sum of the non trivial (resp. trivial) representations
$\rho _{i}$. Then $\mathrm{H}^{*}_{V} \mathrm{S}^{d-1} \cong \mathrm{H}^{*}V \oplus \Sigma^{\mathrm{d}im \tau}\; e(\sigma)\mathrm{H}^{*}V$ and
$\mathrm{H}^{\ast}_{V}(\mathrm{S}^{d-1})$ is an illustration of the remark 6.2.

\section{Determination of $E$ when $V$ is $ \mathbb{Z}/2 \mathbb{Z}$ and $\overline{E}$ is $\mathrm{J}(2) $}
In this section, we assume that $V$ is $ \mathbb{Z}/2 \mathbb{Z}$ and $\overline{E}$ is the Brown-Gitler module $\mathrm{J}(2)$.\\
We denote by $\mathrm{H}= \mathbb{F}_{2}[t]$ the cohomology of $\mathbb{Z}/2\mathbb{Z}$, where $t$ is an element of $\mathrm{H}$
of degree one. We have the following result.\\

\begin{prop} Let $E$ be an $ \mathrm{H}-A$-module which is $ \mathrm{H}$-free and such that $\overline{E}$
is isomorphic to $ \mathrm{J}(2)$ then:\\
$ E \cong \mathrm{H} \otimes  \mathrm{J}(2)$\\
or \\
$E$ is the sub-$\mathrm{H}-A$-module of $\mathrm{H} \oplus \sum \mathrm{H}$ generated by $(t, \Sigma1)$ and $(t^{2},0)$.
\end{prop}
\begin{proof} This proof uses the Smith theory (see \cite{DW}, \cite{LZ2} theorem 2.1) which gives us an exact sequence (*) in $ \mathrm{H}-\mathcal{U}$:
$$
  (*) \;\;\;\;  \xymatrix{0 \ar[r] & E
    \ar[r]^-{\eta}& \mathrm{H} \otimes \mathrm{F}ixE \ar[r]^-{ } & C \ar[r] & 0}
$$
where $C$ the quotient of $\mathrm{H} \otimes \mathrm{F}ixE$ is finite and also $\mathrm{F}ixE$ is finite.\\
If the module $C$ is trivial then $E$ is isomorphic to $\mathrm{H} \otimes \mathrm{J}(2)$.\\
When $C$ is a non trivial module. By applying the functor $ \mathbb{F}_{2} \otimes_{\mathrm{H}}-$ to the exact sequence (*), we obtain:
$$
    \xymatrix{0 \ar[r] &\sum \tau C
    \ar[r]^-{}& \overline{E}= \mathrm{J}(2) \ar[r]^-{ } & \mathrm{F}ixE \ar[r]^-{ } & \overline{C}  \ar[r] & 0}
$$
where $\tau C$ is the trivial part of $C$ (see \cite{BHZ}).\\
Let's denote by $Q$ the quotient of $ \overline{E}$ by $\sum \tau C$. By properties of the module $\mathrm{J}(2)$, we have that
$\sum \tau C = \sum^{2} \mathbb{F}_{2}$ and $Q = \sum \mathbb{F}_{2}$. The exact sequence:
$$
    \xymatrix{0 \ar[r] &\sum \mathbb{F}_{2}
    \ar[r]^-{}&  FixE \ar[r]^-{ } & \overline{C}  \ar[r] & 0}
$$
gives that $FixE \cong \sum \mathbb{F}_{2} \oplus \overline{C}$. One checks that
the module $ \overline{C}$ is either isomorphic to $ \mathbb{F}_{2}$ or $\sum \mathbb{F}_{2}$. If
$ \overline{C}=\sum \mathbb{F}_{2}$ then
$\mathrm{F}ixE \cong \sum \mathbb{F}_{2} \oplus \sum \mathbb{F}_{2}$ as an unstable $A$-module, which implies that the module $E$ is a suspension
which is impossible because $ \overline{E}= \mathrm{J}(2)$ is not a suspension. We conclude that $\overline{C}= \mathbb{F}_{2}$.
Since $\tau C = \sum \mathbb{F}_{2}$ then  we get $C$ is isomorphic to $ \mathrm{H}^{\leq 1}$, where $\mathrm{H}^{\leq 1}$ denotes
the sub-$\mathrm{H}-A$-module of $\mathrm{H}$
consisting of elements of degree less or equal than 1. We have the following exact sequence in
$ \mathrm{H}-\mathcal{U}$:
$$
    \xymatrix{0 \ar[r] & E
    \ar[r]^-{}& \mathrm{H} \oplus \sum \mathrm{H}  \ar[r]^-{\varphi } & \mathrm{H}^{\leq 1}\ar[r] & 0}.
$$
The module $E$, we are searching for, is the kernel of $\varphi$ and we check that it is the sub-$ \mathrm{H}- \mathrm{A}$-module
of $\mathrm{H} \oplus \sum \mathrm{H}$ generated by the elements $(t, \Sigma1)$ and $(t^{2},0)$.
\end{proof}

\begin{rem} {\em Let be $\mathbb{Z}/2\mathbb{Z}$ act on a real projective
space $\mathbb{R}\mathrm{P}^{2}$; let $x_{0}$ be a fixed point of this action (the set of fixed point is not empty for example by an argument
of Lefschetz number). We have:\\
- The Serre spectral sequence collapses to give that: $ \mathrm{H}^{\ast}_{V} (\mathbb{R}\mathrm{P}^{2}, x_{0})$ is $\mathrm{H}$-free and
$\overline{ \mathrm{H}^{\ast}_{V} (\mathbb{R}\mathrm{P}^{2}, x_{0})}$ is isomorphic to $\mathrm{J}(2)$.\\
- In \cite{DW}, Dwyer and Wilkerson have shown that $ \mathrm{H}^{\ast}_{V} \mathbb{R}\mathrm{P}^{2}= \mathbb{F}_{2}[t,y]/(f)$
where $y$ restricts to $x$ and $f=y^{i}(y+t)^{j}$ for
$i+j=3$. It is easy to check that this computation agrees with theorem 7.1.}
\end{rem}


\vspace*{5mm}
\begin{thebibliography} {MML}

\bibitem[BHZ]{BHZ} \textsc{D.Bourguiba, S.Hammouda, S.Zarati}: Profondeur
et cohomologie \'equivariante, African Diaspora Mathematics Research, Special Issue
Vol 4 Number 3, 11-21.

\bibitem[De]{De} \textsc{F.X.Dehon} Cobordisme complexe
des espaces profinis et foncteur $\mathrm{T}$ de Lannes,
M\'{e}moires de la Soci\'{e}t\'{e} Math\'{e}matique de France
\textbf{98}, SMF 2004.


\bibitem[DL]{DL} \textsc{F.X.Dehon, J.Lannes}: Sur les espaces
fonctionnels dont la source est le classifiant d'un groupe de Lie compact,
commutatif I.H.E.S. 89 (1999) 127-177.


\bibitem[DW]{DW} \textsc{W.G.Dwyer, C.W.Wilkerson}: Smith theory revisited, Annals
of Mathematics, 127(1988) 191-198.

\bibitem[EP]{EP} \textsc{M.J.Errockh, C.Peterson}: Injective resolutions of unstable
modules, Journal of Pure and Applied Algebra 97(1994) 37-50.

\bibitem[Hu]{Hu} \textsc{D.Husemoller}: Fibre bundles, McGraw-Hill, series in higher mathematics, 1966.

\bibitem[L1]{L1} \textsc{J.Lannes}: Sur les espaces fonctionnels dont la source est le classifiant d'un
p-groupe ab\'elien \'el\'ementaire, Publ. I.H.E.S. 75 (1992) 135-224.

\bibitem[LS]{LS} \textsc{J.Lannes, L.Shwartz}: Sur la structure des
$A$-modules instables injectifs, Topology 28 (1989) 153-169.

\bibitem[LZ1]{LZ1} \textsc{J.Lannes, S.Zarati}: Sur les
$\mathcal{U}$-injectifs, Ann. Scient. Ec. Norm. Sup. 19 (1986) 1-31.

\bibitem[LZ2]{LZ2} \textsc{J.Lannes, S.Zarati}: Th\'eorie de Smith
alg\'ebrique et classification des $\mathrm{H}^{\ast}V-\mathcal{U}$-injectifs,
Bull. Soc. Math. France 123 (1995) 189-223.

\bibitem[LZ3]{LZ3} \textsc{J.Lannes, S.Zarati}: Tor et Ext-dimensions des $\mathrm{H}^{*}V-\mathrm{A}$-modules
instables qui sont de type fini comme $\mathrm{H}^{*}V$-modules, Progress in Mathematics, Birkh\"{a}user Verlag, vol 136 (1996) 241-253.

\bibitem[NS]{NS} \textsc{M.D.Neusel, L.Smith}: Invariant theory of finite groups, volume 94 of Mathematical
Surveys and Monographs. American Mathematical Society, Providence, RI, 2002.

\bibitem[R]{R} \textsc{J.Rotman}: An introduction to homological algebra, Academic Press, 1979.

\bibitem[S]{S} \textsc{L.Schwartz}: Unstable modules over the Steenrod
algebra and Sullivan's fixed point set conjecture, University of Chicago
Press, 1984.

\bibitem[Se]{Se} \textsc{J-P.Serre}: Sur la dimension cohomologique des groupes profinis,
Topology 3. (1965), 413-420.

\end{thebibliography}
\end{document}